\documentclass[12pt]{amsart}
\usepackage[left=1.5in, right=1.5in, top=1.5in, bottom=1.5in]{geometry}                % See geometry.pdf to learn the layout options. There are lots.
\newtheorem{thm}{Theorem}[section]
\newtheorem{cor}[thm]{Corollary}

\theoremstyle{definition}

\newtheorem{rem}[thm]{Remark}
\usepackage{graphicx}
\usepackage{amssymb}
\usepackage{amsmath}
\usepackage{epstopdf}
\usepackage{hyperref}
\title[A WEIGHTED RELATIVE ISOPERIMETRIC INEQUALITY]{A WEIGHTED RELATIVE ISOPERIMETRIC INEQUALITY IN CONVEX CONES
}
\author{Emanuel Indrei}
\address{Department of Mathematics\\
Purdue University\\
West Lafayette, Indiana \\
USA.}

\begin{document}
\setcounter{page}{1}
\pagenumbering{arabic}
\maketitle
\begin{abstract}
A weighted relative isoperimetric inequality in convex cones is obtained via the Monge-Ampere equation. The method improves several inequalities in the literature, e.g. constants in a theorem of Cabre--Ros--Oton--Serra. Applications are given in the context of a generalization of the log-convex density conjecture due to Brakke and resolved by Chambers: in the case of $\alpha-$homogeneous ($\alpha>0$), concave densities, (mod translations) balls centered at the origin and intersected with the cone are proved to uniquely minimize  the weighted perimeter with a weighted mass constraint. In particular, if the cone is taken to be $\{x_n>0\}$, reflecting the density, balls intersected with $\{x_n>0\}$ remain (mod translations) unique minimizers in the $\mathbb{R}^n$ analog in the case when the density vanishes on $\{x_n=0\}$. 
\end{abstract}

\vskip .5in

Suppose $\mathcal{C} \subset \mathbb{R}^n$ is an open, convex cone ($n \ge 2$). If $|E|<\infty$ is a set of finite perimeter with reduced boundary $\mathcal{F} E$, $K = B_1 \cap \mathcal{C}$, then
$$
n|E|^{\frac{n-1}{n}}|K|^{\frac{1}{n}} \le \mathcal{H}^{n-1}(\mathcal{F} E \cap \mathcal{C})
$$
where equality holds if and only if $E =$ mod translations (if the cone contains lines) $sK$ for $s > 0$. If $|E| = |K|$, the inequality states

$$
n|E| \le \mathcal{H}^{n-1}(\mathcal{F}E \cap \mathcal{C})
$$
and this via scaling is equivalent; a sharp stability result for this inequality
was proved in Figalli-Indrei \cite{MR3023863}: the simple version states that if $\mathcal{C}$ contains no line

$$
|E\Delta K| \lesssim \Big(\frac{\mathcal{H}^{n-1}(\mathcal{F} E \cap \mathcal{C})}{n|E|}-1 \Big)^{\frac{1}{2}}
$$
($E\Delta K$ is the symmetric difference). In the theorem below, a weighted version of the relative isoperimetric inequality in convex cones is shown. The proof is based on the Monge-Ampere equation \cite{MR1426885, MR1177479, MR1454261, TrW}

$$
\det D^2 \phi = f.
$$

\begin{thm} \label{OT}
If $\mathcal{C} \subset \mathbb{R}_{+}^n$ is an open convex cone, $E \subset \mathcal{C}$ is a set of finite perimeter with $|E|=|K|$, $h(x) \ge 0$, $
\inf_{a\in K} \nabla h \cdot a \ge 0
$, then 
$$
\int_E nh(x)dx \le \int_{\mathcal{F} E \cap \mathcal{C}} h(x) d \mathcal{H}^{n-1};
$$
moreover, if $\mathcal{C}\subsetneq \mathbb{R}_{+}^n$, $h=h(x_n)>0$ for $x_n>0$, $\&$ equality holds, then (up to sets of measure zero) $E=K$ (if $E=K$, then 
$$
\int_K nhdx+\int_K\nabla h \cdot xdx=\int_{\partial K \cap \mathcal{C}} hd\mathcal{H}^{n-1}).
$$
\end{thm}

\begin{proof}
Let $d\mu^+=\chi_Edx$, $d\mu^-=\chi_Kdy$ and $T_{\#}\mu^+=d\mu^-$ denote the optimal map (i.e. $T=\nabla \phi$, $\phi$ convex) then for a.e. $x \in E$
$$
\det DT(x)=1.
$$
Thus, 
\begin{align*}
\int_E nhdx&= \int_E n\big(\det DT\big)^{\frac{1}{n}} hdx\\
&\le \int_E (\text{divT})hdx=\int_E \text{divTh} dx- \int_E \nabla h \cdot \text{T}dx\\
&\le \int_E \text{divTh} dx\\
&\le \int_{E^{(1)}} \text{divTh}dx+(\text{DivTh})_s(E^{(1)})\\
&=\text{DivTh}(E^{(1)})\\
&= \int_{\mathcal{F}E} \text{tr}_E(Th) \cdot \nu_E(x) d\mathcal{H}^{n-1}\\
&\le \int_{\mathcal{F}E \cap \mathcal{C}} h d \mathcal{H}^{n-1}
\end{align*}
($E^{(1)}$ is the set of density $1$, $\text{(DivTh)}_s$ is the singular part of the measure
DivTh, and $\text{tr}_E$ is the trace, see \cite{MR3023863}). Let
\begin{equation} \label{er}
\int_E nh(x_n)dx=\int_{\mathcal{F}E \cap \mathcal{C}} h(x_n) d\mathcal{H}^{n-1};
\end{equation}
in particular, there is a.e. equality in the arithmetic-geometric mean inequality which yields $DT=\text{Id}$ a.e. and $$\int_{\mathcal{F}E} \text{tr}_E(Th) \cdot \nu_E(x) d\mathcal{H}^{n-1}\\
= \int_{\mathcal{F}E \cap \mathcal{C}} h d \mathcal{H}^{n-1}$$
implies $T(x)=x+x_0$ for $x_0 \in \mathbb{R}^n$. Therefore $E$ is connected and since 
$$
\int_E T_n h'(x_n)dx=0,
$$
$h'=0$ a.e. on $E$; in particular $h$ is constant on $E$ and \eqref{er} implies $x_0=0$. (If $E=K$, then $T(x)=x.$)
\end{proof}
\begin{rem}
In the case when $\mathcal{C}$ contains a line, equality holds up to translations along the line (a general convex cone splits: $\mathcal{C}=\mathcal{C}_k \times \mathbb{R}^{n-k}$, where $\mathcal{C}_k\subset \mathbb{R}^k$ is a convex cone containing no line \cite{MR3023863, MR3385175}).
\end{rem}
\begin{rem}
In the function space $\{f: f=h\chi_E, \inf_{a\in K} \nabla h \cdot a \ge 0, |E|=|K|\}$, the inequality is sharp: equality is attained for $f=\alpha \chi_K$, $\alpha \in \mathbb{R}$.
\end{rem} 
\begin{rem} \label{le}
Weighted isoperimetric and/or Sobolev inequalities were studied in \cite{MR3576542, MR3268875, balogh2020sobolev, MR3228444, MR3005102, alvino2019isoperimetric, MR4056843, MR4011033, MR2948294, MR4135671, MR2466858, MR3038539, MR3097258, MR2645984, cinti2020sharp, MR1387625, MR1854254, MR1396954}. Furthermore, I recently proved the following result which answers a long-standing open problem due to Almgren (the anisotropic isoperimetric problem is considered in a convex background) \cite{indrei2020equilibrium}:
let $$f:\mathbb{R}^2\rightarrow [0,\infty) $$ be a surface tension (a convex positively 1-homogeneous function), 

$$
\mathcal{P}(E)=\int_{\mathcal{F}E} f(\nu_E) d\mathcal{H}^{n-1};
$$
$$
\mathcal{G}(E)=\int_E g(x)dx;
$$ 
$$
\mathcal{E}(E)=\mathcal{P}(E)+\mathcal{G}(E).
$$
Suppose $g:\mathbb{R}^2 \rightarrow [0,\infty)$ is coercive and convex. If $m\in (0,\infty)$, there exists an--up to translations $\&$ sets of measure zero--unique convex set which minimizes the free energy $\mathcal{E}$ among sets with measure $m$ (inter-alia I proved non-quantitative stability). Moreover, there exists a convex g such that there are no solutions for any $m>0$.

\end{rem}
\begin{cor}
If $\mathcal{C} \subset \mathbb{R}_{+}^n$ is an open convex cone, $E_\epsilon \subset \mathcal{C}$ is a set of finite perimeter with $|E_\epsilon|=|K|$, $h(x) \ge 0$, 
$$
\int_K \nabla h \cdot x dx > 0,
$$ 

$$ 
E_\epsilon \rightarrow K
$$
in $L^1_{loc}$,
then there exists $a(\epsilon)>0$, such that $a(\epsilon) \rightarrow 0$ as $\epsilon \rightarrow 0^+$ (mod a subsequence), $\&$ 
$$
-a(\epsilon) \int_{\partial K \cap \mathcal{C}} hd\mathcal{H}^{n-1} \le \int_{\partial E_\epsilon \cap \mathcal{C}} h d \mathcal{H}^{n-1}-\int_{\partial K \cap \mathcal{C}} h d \mathcal{H}^{n-1}.
$$
Hence, if $\frac{a(\epsilon)}{\alpha_\epsilon} \rightarrow 0^+,$
$$
\lim \inf_{\epsilon \rightarrow 0^+} \frac{1}{\alpha_\epsilon} \Big(\int_{\partial E_\epsilon \cap \mathcal{C}} h d \mathcal{H}^{n-1}-\int_{\partial K \cap \mathcal{C}} h d \mathcal{H}^{n-1} \Big)\ge 0.
$$

\end{cor}

\begin{proof}
Let $d\mu^+_\epsilon=\chi_{E_\epsilon}dx$, $d\mu^-=\chi_Kdy$ and $(T_\epsilon)_{\#}\mu^+_\epsilon=d\mu^-$ denote the optimal map. Since $T_\epsilon \chi_{E_\epsilon} \rightharpoonup x \chi_K$ (along a subsequence)
$$
\Big|\int_{E_\epsilon} T_\epsilon \cdot \nabla h dx- \int_K x \cdot \nabla h dx \Big| \le a(\epsilon) \int_K \nabla h \cdot x dx
$$
for $a(\epsilon) \rightarrow 0$ ($ \int_{E_\epsilon} |T_\epsilon|^2 dx= \int_{K} |x|^2 dx, \hskip .05in \lim\sup_{\epsilon \rightarrow 0} \int_{E_\epsilon} |x|^2 dx<\infty$).
Hence for $\epsilon$ sufficiently small, 
$$
\int_{E_\epsilon} T_\epsilon \cdot \nabla h dx \ge (1-a(\epsilon)) \int_K x \cdot \nabla h dx,
$$
$$
n\int_{E_\epsilon} h dx \ge (1-a(\epsilon))n \int_K h dx, 
$$
and the proof of the theorem implies the inequality.
\end{proof}
\begin{rem}
If $|E|\neq|K|$, the previous arguments apply to $a E$ such that $|a E|=|K|$.
\end{rem}

The log-convex density conjecture is stated in terms of the following formulation: if $h: \mathbb{R}^n \rightarrow [0,\infty)$ is radially log-convex with a smooth density ($h(x)=e^{\phi(|x|)}$, $\phi$ convex, smooth), then a ball $B$ around the origin such that $\int_B h dx=m$ solves
$$
\inf_{\int_E h dx=m} \int_{\mathcal{F} E} h d\mathcal{H}^{n-1}.
$$
There is a large collection of papers in the literature on this conjecture, e.g. \cite{MR2342613, MR4035846, MR3864811, MR3116018, MR2858470, MR3581260} (see also Remark \ref{le}), and results vary in terms of the assumptions on $\phi$: if $\phi \in C^3(0,\infty)$, the above was proved by Chambers (McGillivray obtained this in $\mathbb{R}^2$ assuming $\phi$ was non-decreasing, convex). The next theorem yields the analog in convex cones for $\alpha-$homogeneous, concave functions $h$ and characterizes the minimizers--up to translations--uniquely (see Corollary \ref{a1}).

\begin{thm} \label{alpha}
If $\mathcal{C} \subset \mathbb{R}_{+}^n$ is an open convex cone, $E \subset \mathcal{C}$ is a set of finite perimeter, $h(x) \ge 0$ is concave and $\alpha-$homogeneous, $\alpha \ge 0$, then 
$$
(n+\alpha-1) \big(\frac{|K|}{|E|}\big)^{\frac{1}{n}} \int_E h dx+\frac{1}{\big(\frac{|K|}{|E|}\big)^{\frac{n+\alpha-1}{n}}}\int_K h dx  \le  \int_{\mathcal{F} E \cap \mathcal{C}} h d \mathcal{H}^{n-1}.
$$
Suppose $h(x)>0$ for some $x \in \mathcal{C}$ and that equality holds, then $aE=$ (mod translations) $K$.
\end{thm}

\begin{proof}
If $\alpha=0$, $h=\text{constant}$, and the claim follows via the relative isoperimetric inequality in convex cones; therefore without loss of generality, $\alpha>0$. Let $|E|=|K|$, $d\mu^+=\chi_Edx$, $d\mu^-=\chi_Kdy$, and $T_{\#}\mu^+=d\mu^-$ denote the optimal map (i.e. $T=\nabla \phi$, $\phi$ convex) then for a.e. $x \in E$
$$
\det DT(x)=1.
$$
Note that since $h$ is $\alpha-$homogeneous, $\alpha h(x)=\nabla h(x) \cdot x$ and concavity yields 
$$
\nabla h(x) \cdot T=\nabla h(x) \cdot (T-x)+\alpha h(x) \ge h(T)-h(x)(1-\alpha).
$$ 
Thus
\begin{align*}
\int_E nhdx&= \int_E n\big(\det DT\big)^{\frac{1}{n}} hdx\\
&\le \int_E (\text{divT})hdx=\int_E \text{divTh} dx- \int_E \nabla h \cdot \text{T}dx\\
&\le \int_E \text{divTh} dx -\int_K h dx+(1-\alpha)\int_E h(x)dx\\
&\le \int_{E^{(1)}} \text{divTh}dx+(\text{DivTh})_s(E^{(1)})-\int_K h dx+(1-\alpha)\int_E h(x)dx\\
&=\text{DivTh}(E^{(1)})-\int_K h dx+(1-\alpha)\int_E h(x)dx\\
&= \int_{\mathcal{F}E} \text{tr}_E(Th) \cdot \nu_E(x) d\mathcal{H}^{n-1}-\int_K h dx+(1-\alpha)\int_E h(x)dx\\
&\le \int_{\mathcal{F}E \cap \mathcal{C}} h d \mathcal{H}^{n-1}-\int_K h dx+(1-\alpha)\int_E h(x)dx
\end{align*}
($E^{(1)}$ is the set of density $1$, $\text{(DivTh)}_s$ is the singular part of the measure
DivTh, and $\text{tr}_E$ is the trace). 
Thus
$$
(n+\alpha-1) \int_E h dx+\int_K h dx  \le  \int_{\mathcal{F} E \cap \mathcal{C}} h d \mathcal{H}^{n-1}.
$$
In the case of equality, if $\mathcal{C}$ does not contain a line note that  
$$
\nabla h(x) \cdot T= h(T)-h(x)(1-\alpha)
$$  
a.e., and thus in the case of strictly concave densities $h$, $T(x)=x$ a.e.; if $\mathcal{C}$ contains a line, equality holds up to translations; in the case when $h$ is concave but not strictly concave, 
$$
\int_E n\big(\det DT\big)^{\frac{1}{n}} hdx=\int_E (\text{divT})hdx;
$$
in particular, there is a.e. equality in the arithmetic-geometric mean inequality which yields $DT=\text{Id}$ a.e. and 
$$\int_{\mathcal{F}E} \text{tr}_E(Th) \cdot \nu_E(x) d\mathcal{H}^{n-1}
= \int_{\mathcal{F}E \cap \mathcal{C}} h d \mathcal{H}^{n-1}$$
implies $T(x)=x+x_0$ for $x_0 \in \mathbb{R}^n$.  
Suppose $|E|\neq |K|$ and let $|aE|=|K|$. The homogeneity implies  
$$
\int_{aE} hdx=a^{n+\alpha} \int_E hdx,
$$
$$
\int_{\mathcal{F} (aE) \cap \mathcal{C}} h d \mathcal{H}^{n-1}=a^{n+\alpha-1}\int_{\mathcal{F} E \cap \mathcal{C}} h d \mathcal{H}^{n-1},
$$
thus
$$
(n+\alpha-1) \big(\frac{|K|}{|E|}\big)^{\frac{1}{n}} \int_E h dx+\frac{1}{\big(\frac{|K|}{|E|}\big)^{\frac{n+\alpha-1}{n}}}\int_K h dx  \le  \int_{\mathcal{F} E \cap \mathcal{C}} h d \mathcal{H}^{n-1}.
$$

\end{proof}
\begin{rem}
In the case of strictly concave $\alpha-$homogeneous densities $h \ge 0$, if $\mathcal{C}\subsetneq \mathbb{R}_{+}^n$ and equality holds, then $aE=K.$
\end{rem}

\begin{cor} \label{a1}
If $\mathcal{C} \subset \mathbb{R}_{+}^n$ is an open convex cone, $E \subset \mathcal{C}$ is a set of finite perimeter, $h(x) \ge 0$ is concave and $\alpha-$homogeneous, $\alpha \ge 0$, then 
if $\int_E h dx=\int_K h dx$, 
$$
\frac{1}{n+\alpha}\Big((n+\alpha-1) \big(\frac{|K|}{|E|}\big)^{\frac{1}{n}}+\frac{1}{\big(\frac{|K|}{|E|}\big)^{\frac{n+\alpha-1}{n}}}\Big)\int_{\partial K \cap \mathcal{C}} hd \mathcal{H}^{n-1} \le  \int_{\mathcal{F} E \cap \mathcal{C}} h d \mathcal{H}^{n-1},
$$
and equality holds if and only if $aE=$ (mod translations) $K$, $a>0$;
in particular
$$
\int_{\partial K \cap \mathcal{C}} h d \mathcal{H}^{n-1}=\inf_{\int_E h dx=\int_K h dx} \int_{\mathcal{F} E \cap \mathcal{C}} h d \mathcal{H}^{n-1}
$$
and the infimum is attained uniquely by $E=$ (mod translations) $K$.
\end{cor}

\begin{proof}
First, the homogeneity implies (without loss $\alpha>0$)
$$
\int_K h dx = \frac{1}{n+\alpha} \int_{\partial K \cap \mathcal{C}} h d \mathcal{H}^{n-1},
$$
and the first result follows from the theorem. Set  
$$
g(a)=k a+\frac{1}{a^k},
$$
$a\ge 0$, $k=n+\alpha-1$; a simple calculation shows that the minimum is attained at $a=1$, and this yields 
$$
\int_{\partial K \cap \mathcal{C}} h d \mathcal{H}^{n-1} \le \int_{\mathcal{F} E \cap \mathcal{C}} h d \mathcal{H}^{n-1}
$$ when 
$\int_E h dx=\int_K h dx$. 
Suppose equality is attained, then
$$\frac{1}{n+\alpha}\Big((n+\alpha-1) \big(\frac{|K|}{|E|}\big)^{\frac{1}{n}}+\frac{1}{\big(\frac{|K|}{|E|}\big)^{\frac{n+\alpha-1}{n}}}\Big)=1,$$ and
this implies $|E|=|K|$; hence $E=$ (mod translations) $K$ by the first claim. 
\end{proof}
The inequality in Corollary \ref{a1} implies a quantitative estimate, cf. \cite{cinti2020sharp, MR3487241}.
\begin{cor}
If $\mathcal{C} \subset \mathbb{R}_{+}^n$ is an open convex cone, $E \subset \mathcal{C}$ is a set of finite perimeter, $h(x) \ge 0$ is concave and $\alpha-$homogeneous, $\alpha \ge 0$, then 
if $\int_E h dx=\int_K h dx$, 
$$
 \int_{\mathcal{F} E \cap \mathcal{C}}h d \mathcal{H}^{n-1}-\int_{\partial K \cap \mathcal{C}} hd \mathcal{H}^{n-1} \ge \int_{\partial K \cap \mathcal{C}} hd \mathcal{H}^{n-1}\Big(1- \frac{1}{n+\alpha}\Big((n+\alpha-1) \big(\frac{|K|}{|E|}\big)^{\frac{1}{n}}+\frac{1}{\big(\frac{|K|}{|E|}\big)^{\frac{n+\alpha-1}{n}}}\Big)\Big),
$$
and 
$$
\int_{\mathcal{F} E \cap \mathcal{C}}h d \mathcal{H}^{n-1}=\int_{\partial K \cap \mathcal{C}} hd \mathcal{H}^{n-1}
$$
if and only if $E=$ (mod translations) $K$.
\end{cor}
\begin{rem}
$\alpha$-homogeneous functions appear as blow-up limits in free boundary problems \cite{MR4085452, MR3986537, MR3957397}.
\end{rem}

\begin{cor}
If $h(x) \ge 0$ is concave, $\alpha-$homogeneous for $x \in \{x_n>0\}$ ($\alpha > 0$), $h(x',0)=0$, $h(x',x_n)=h(x',-x_n)$, then 
$$
\int_{\partial B_R \cap \{x_n>0\}} h d \mathcal{H}^{n-1}=\inf_{\int_E h dx=\int_{B_1} h dx, E \subset \mathbb{R}^n } \int_{\mathcal{F} E } h d \mathcal{H}^{n-1},
$$
$R=2^{\frac{1}{n+\alpha}}$ and the infimum is attained in the collection of sets of finite perimeter uniquely by $E=$ (mod translations) $B_R \cap \{x_n>0\}$.

\end{cor}

\begin{proof}
If $E \subset \mathbb{R}^n$ is a minimizer, let 
$E^{+}=E \cap \{x_n>0\}$, $E^{-}=E \cap \{x_n<0\}$, 
$$
\int_{E^+} h dx+\int_{E^-} h dx=\int_{B_1} h dx.
$$
The theorem implies
$$
\int_{\partial E^+\cap \{x_n>0\}} hd \mathcal{H}^{n-1} +\int_{E^-\cap \{x_n<0\}} h d \mathcal{H}^{n-1}\ge \int_{\partial B_{R_1}\cap \{x_n>0\}} h d \mathcal{H}^{n-1}+\int_{\partial B_{R_2}\cap \{x_n<0\}} h d \mathcal{H}^{n-1},
$$
$$
\int_{B_{R_1}\cap \{x_n>0\}} h dx=\int_{E^+} h dx, \int_{B_{R_2}\cap \{x_n<0\}} h dx=\int_{E^-} h dx; 
$$
in particular, 
$$
\int_{\partial E} hd \mathcal{H}^{n-1}\ge \frac{1}{2}\Big(R_1^{n+\alpha-1}+R_2^{n+\alpha-1}\Big)\int_{\partial B_1} h d \mathcal{H}^{n-1},
$$
$$
\frac{1}{2}\Big( R_1^{n+\alpha}+R_2^{n+\alpha}\Big)=1;$$ and
minimizing 
$$
R_1^{n+\alpha-1}+R_2^{n+\alpha-1}
$$
subject to 
$$
\frac{1}{2}\Big( R_1^{n+\alpha}+R_2^{n+\alpha}\Big)=1
$$
implies $R_1=2^{\frac{1}{n+\alpha}}, R_2=0$ (without loss). 

\end{proof}
\begin{rem}
Let $h(x,y)=y$  if $y \ge 0$ and extend via even reflection. Note that 

$$
\int_{B_{R}^+}hdxdy=R^3\int_{B_{1}^+}hdxdy=R^3 \int_0^\pi \int_0^1 \sin \theta rdrd\theta =R^3; 
$$

$$
\int_{B_{R}}hdxdy=2R^3;
$$

$$
\int_{\partial B_{R} \cap \{y>0\}}hd \mathcal{H}^{1}=R^2\int_{\partial B_{1}\cap \{y>0\}}hd \mathcal{H}^{1}=R^2 \int_0^\pi \sin \theta d\theta= 2R^2; 
$$

$$
\int_{\partial B_{R} }hd \mathcal{H}^{1}=4R^2; 
$$
this implies 
$$
\int_{B_{R_*}}hdxdy=2R_{*}^3=\int_{B_{R}^+}hdxdy=R^3
$$
if and only if
$$
R_{*}=\frac{R}{2^{\frac{1}{3}}};
$$
thus

$$
\int_{\partial B_{R_*} }hd \mathcal{H}^{1}=4R_*^2=\frac{2}{2^{\frac{2}{3}}}\int_{\partial B_{R} \cap \{y>0\}}hd \mathcal{H}^{1}>\int_{\partial B_{R} \cap \{y>0\}}hd \mathcal{H}^{1}.
$$

\end{rem}

In the next corollary, the ``concave--$\alpha$ analog" of Theorem \ref{alpha} is stated which is a strict improvement of a ``concave--$\frac{1}{\alpha}$ analog" of a  theorem in \cite{MR3576542}, see Remark \ref{a5}; the equality cases have only recently appeared on arXiv \cite{cinti2020sharp} after the first version of this paper was submitted.

\begin{cor} \label{a2}

If $\mathcal{C} \subset \mathbb{R}_{+}^n$ is an open convex cone, $E \subset \mathcal{C}$ is a set of finite perimeter, $h(x) \ge 0$ is concave and $\alpha-$homogeneous, then 

$$
\int_{\partial K \cap \mathcal{C}} hdx \le \frac{(n+\alpha)\big(\frac{\int_K hdx}{\int_E hdx}\big)\big(\frac{|K|}{|E|}\big)^{\frac{n+\alpha-1}{n}}}{\big(\frac{\int_K hdx}{\int_E hdx}\big)+(n+\alpha-1)\big(\frac{|K|}{|E|}\big)^{\frac{n+\alpha}{n}}} \int_{\mathcal{F} E \cap \mathcal{C}} h d \mathcal{H}^{n-1}.
$$
Suppose $h(x)>0$ for some $|x|\neq 0$ and that equality holds, then $aE=$ (mod translations) $K$.

\end{cor}

\begin{proof}
The inequality is immediate via Theorem \ref{alpha}. First, assume $|E|=|K|$ and that equality holds; then 
\begin{equation*} 
\int_E n\big(\det DT\big)^{\frac{1}{n}} hdx=\int_E (\text{divT})hdx;
\end{equation*}
in particular, there is a.e. equality in the arithmetic-geometric mean inequality which yields $DT=\text{Id}$ a.e. and $$\int_{\mathcal{F}E} \text{tr}_E(Th) \cdot \nu_E(x) d\mathcal{H}^{n-1}
= \int_{\mathcal{F}E \cap \mathcal{C}} h d \mathcal{H}^{n-1}$$
implies $T(x)=x+x_0$ for $x_0 \in \mathbb{R}^n$. If $|E|\neq |K|$, let $|aE|=|K|$;  the result follows via the previous argument applied to $aE$.
\end{proof}

\begin{rem} \label{a5}
In \cite{MR3576542}, the authors obtain the following:
if $\mathcal{C} \subset \mathbb{R}_{+}^n$ is an open convex cone, $E \subset \mathcal{C}$ is a set of finite perimeter, $h(x) \ge 0$ is $\alpha-$homogeneous $\&$ $h^{\frac{1}{\alpha}}$ is concave, then 
$$
\int_{\partial K \cap \mathcal{C}} hdx \le \Big(\frac{\int_K hdx}{\int_E hdx}\Big)^{\frac{n+\alpha-1}{n+\alpha}} \int_{\mathcal{F} E \cap \mathcal{C}} h d \mathcal{H}^{n-1}.
$$
Define $\overline{f}(x,y):=\frac{(n+\alpha)xy^{\frac{n+\alpha-1}{n}}}{x+(n+\alpha-1)y^{\frac{n+\alpha}{n}}}$ for $x,y>0$. Then whenever the function $h$  is concave, $\alpha-$homogeneous and $h^{\frac{1}{\alpha}}$ is concave (e.g. $0<\alpha \le1$) 
Corollary \ref{a2} improves the inequality if and only if 
$$
\overline{f}(x,y) < x^{\frac{n+\alpha-1}{n+\alpha}}.
$$  
Set 
$$f(x,y)=x^{-\frac{1}{n+\alpha}}-\frac{(n+\alpha)y^{\frac{n+\alpha-1}{n}}}{x+(n+\alpha-1)y^{\frac{n+\alpha}{n}}}
=x^{-\frac{1}{n+\alpha}}-\frac{(n+\alpha)}{xy^{\frac{1-n-\alpha}{n}}+(n+\alpha-1)y^{\frac{1}{n}}};$$
fix $y>0$ and note that 
$$
f(x,y) \rightarrow \infty, \hskip .2in x \rightarrow 0^+;
$$
$$
f(x,y) \rightarrow 0, \hskip .2in x \rightarrow \infty;
$$
if $x>0$ is fixed,
$$
f(x,y) \rightarrow x^{-\frac{1}{n+\alpha}}, \hskip .2in y \rightarrow 0^+, \infty;
$$
$$
\partial_y f(x,y)=\frac{(n+\alpha)}{n}\Big(\frac{(n+\alpha-1)y^{\frac{1-n}{n}}+(1-n-\alpha)xy^{\frac{1-2n-\alpha}{n}}}{(xy^{\frac{1-n-\alpha}{n}}+(n+\alpha-1)y^{\frac{1}{n}})^2}\Big)
$$
and thus 
$$
\partial_y f(x,y)=0
$$ 
when $y(x)=x^{\frac{n}{n+\alpha}}$ and 

$$
f(x,y(x))=0. 
$$ 
This yields that if $x,y>0$, $f(x,y)\ge 0$ with equality if and only if $y =x^{\frac{n}{n+\alpha}}$. 
If $|E|=|K|$,  
 
$$
f(x,1)=x^{-\frac{1}{n+\alpha}}-\frac{n+\alpha}{x+n+\alpha-1} \ge 0
$$ 
and is equal if and only if $x=1$: if $x>0$
$$
h(x)=x^{\frac{1}{n+\alpha}}=1+\frac{1}{(n+\alpha)}(x-1)+h''(a_x)(x-1)^2 \le 1+\frac{1}{(n+\alpha)}(x-1);  
$$
in particular, when $y(x)\neq x^{\frac{n}{n+\alpha}}$ Corollary \ref{a2} \text{\bf is a strict improvement} of the result in \cite{MR3576542} and is equivalent when $y(x)=x^{\frac{n}{n+\alpha}}$.  
\end{rem}
\begin{rem}
If $\mathcal{C}^{n+1}$ is a manifold with density $h$ (a connected manifold with a Riemannian metric $\langle \cdot, \cdot \rangle$ and smooth positive function weighing the Hausdorff measures associated to the Riemannian distance), then if $h$ is $\alpha-$homogeneous,
$$D^2 h^{\frac{1}{\alpha}} =\frac{-1}{\alpha} h^{\frac{1}{\alpha}} \text{Ric}_h^\alpha=\frac{-1}{\alpha} h^{\frac{1}{\alpha}} \Big(\text{Ric} - D^2 \log h-\frac{1}{\alpha}(d \log h \otimes d \log h)\Big),
$$
where Ric is the Ricci tensor, $D^2$ is the Hessian operator for the Riemannian metric, and $\text{Ric}_h^\alpha$ is the $\alpha$-dimensional Bakry-Emery-Ricci tensor. Therefore in the case when $\alpha>0$, $h^{\frac{1}{\alpha}}$ is concave iff $ \text{Ric}_h^\alpha \ge 0$ \cite[Lemma 3.9]{MR3268875}.  
\end{rem}

{\bf Acknowledgment}

I would like to thank the organizers of the conference ``Workshop on Monge-Ampere equations: in celebration of Professor John Urbas’s 60th birthday" for the invitation to contribute a talk and a paper. Furthermore, I would like to thank Arshak Petrosyan for his question on a weighted version of the relative isoperimetric inequality in convex cones after my lecture at Purdue in November 2019.

\newpage
\bibliographystyle{amsalpha}
\bibliography{References}

\end{document}